\documentclass[a4paper,11pt,twoside]{article}
\usepackage{pst-fill,pst-grad}
\usepackage{textcomp}
\usepackage[english]{babel}
\usepackage{graphicx}
\usepackage{amsmath}
\usepackage{float}
\usepackage[matrix,arrow,curve]{xy}
\usepackage{pstricks}
\usepackage{amsmath,amsfonts,verbatim,afterpage,theorem,euscript,mathrsfs,amssymb}
\usepackage{amsfonts}
\usepackage{amssymb}
\usepackage{eulervm}
\usepackage{array}
\usepackage{dsfont}
\usepackage{hyperref}
\newtheorem{Theoreme}{Theorem}

\newtheorem{Definition}{Definition}[section]
\newtheorem{Proposition}{Proposition}[section]
\newtheorem{Lemme}{Lemma}[section]

\newtheorem{Remarque}{Remark}[section]

\numberwithin{equation}{section}
\newenvironment{proof}{\paragraph{Proof.}}{\hfill$\square$}

\def\vf{f}
\def\vg{\vec{g}}

\def\vn{\vec{\nabla}}

\def\Rt{\mathbb{R}^n}
\def\R{\mathbb{R}}
\def\mP{\mathcal{P}}
\def\mPl{\mathcal{P}^{\text{log}}}

\title{\bf On variable Lebesgue spaces and generalized
nonlinear heat 
equations}
\usepackage{authblk}
\author{Gast\'on Vergara-Hermosilla\footnote{\emph{gaston.vergarahermosilla@univ-evry.fr}   } }
\affil{\footnotesize LaMME, Univ. Evry, CNRS, Universit\'e Paris-Saclay, 91025, Evry, France.}
\begin{document}

	\maketitle

	\begin{abstract}  
In this work  we address some questions concerning  
the Cauchy problem
for a generalized nonlinear
heat equations considering as functional framework the variable Lebesgue spaces  $L^{p(\cdot)}(\R^n)$.
More precisely, by mixing some structural properties of these spaces with decay estimates of the fractional heat kernel, we were able to prove two well-posedness results for these equations.  In a first theorem, we prove the existence and uniqueness of global-in-time  mild solutions in the mixed-space $\mathcal{L}^{p(\cdot)}_{   \frac{nb}{2\alpha -
\langle 1 \rangle_\gamma
}      }
(\mathbb{R}^n,L^\infty([0,T[ ))$. On the other hand,  by introducing a new class of variable exponents, we demonstrate the existence of an unique local-in-time mild solution  in the space $L^{p(\cdot)} \left( [0,T], L^{q} (\R^n) \right)$.


\end{abstract}

\section{Introduction}

\subsection{General setting}
In this paper we study the Cauchy problem  for the generalized nonlinear heat equations 
\begin{equation}\label{NS_Intro}
\begin{cases}
\partial_tu+  (-\Delta)^{\alpha} u
= F(u)
+  {f}, &  (t,x)\in \ ]0,+\infty[\times \R^n \\[4pt]
u(0,x)=u_0(x),  &  x\in \mathbb{R}^n,
\end{cases}
\end{equation}
where $u:[0,+\infty[\times \mathbb{R}^n \longrightarrow \mathbb{R}$, 
and ${f}:[0,+\infty[\times \mathbb{R}^n \longrightarrow \mathbb{R}$ is a given external force. 
Equation (\ref{NS_Intro}) possess two main features, on the one hand
we have involved 
a nonlinear term given by 
\begin{equation}\label{nonlinearity}
    F(u) = |u|^b u
\quad \text{or}\quad 
F(u) = \vec{1}\cdot \vn(|u|^b u ),
\end{equation}
with $b\in \mathbb{N}\setminus \{0\}$ and $\vec{1}:=(1,1,...,1)\in \mathbb{R}^n $, and on the other hand
we considered
the fractional Laplacian operator $(-\Delta)^{\alpha}$  in the diffusion term. Recall that this operator is defined at  the Fourier level by the symbol $|\xi|^{2\alpha}$, whereas in the spatial variable we have 
\[ (-\Delta)^{\alpha} u (t,x)= C_{\alpha}\, {\bf p.v.} \int_{\Rt} \frac{u (t,x)- u (t,y)}{|x-y|^{3+2\alpha}} dy,  \]
where $C_\alpha>0$ is a constant depending on $\alpha$, and  ${\bf p.v.}$ denotes the principal value.
\\

The evolution equation \eqref{NS_Intro} models two of the classical equations in the literature; the generalized semi-linear  power dissipative equation and the generalized convection-diffusion equation. 
The case $\alpha=1$ 
corresponds to the well-known  classical semi-linear heat equation and
has been widely studied by many authors, where we highlight the papers \cite{hiroshifujita,Matosterraneo,ponce,terraneo02,Weissler}   and the classical book \cite{MR3967048}. 
In particular, the case with $\alpha=1$, nonlinear term  $
\vec{1}\cdot \vn(|u|^b u )$,  and $b=1$, is of special interest as it
can be interpreted as  an scalar toy model for the Navier-Stokes equations \cite{MR3808277,MR3983288,MR3469428,MR3308920}. 
Regarding more general cases,   
in \cite{MR1844433,MR1663000} the authors proved for $\alpha\in \mathbb{N}$ 
the existence and uniqueness of strong solutions of \eqref{NS_Intro}
considering as functional setting the classical Lebesgue spaces.
On the other hand, in \cite{MR2290390} the authors provide a global well-posedness result in
the case of $\alpha\geq 1$
considering small initial data in pseudomeasure spaces. \\

In this paper we address questions concerning the existence and uniqueness of solutions for equations (\ref{NS_Intro}) 
considering as functional framework the Lebesgue spaces of variable exponent $L^{p(\cdot)}(\R^n)$.
To the best of our knowledge, this kind of functional spaces have not been considered previously in the analysis of  semi-linear power dissipative type equations nor  convection-diffusion type equations. 
Roughly speaking, the  spaces $L^{p(\cdot)}(\R^n)$ can be interpreted as a natural generalization of the classical Lebesgue spaces $L^p$, in the sense that, 
the usual constant parameter $p\in[1,+\infty[$ is replaced now by an appropriate function $p(\cdot):\R^n\longrightarrow [1,+\infty[$. 
However, as we will see in the following lines, in the rigorous definition of the variable Lebesgue spaces there are subtle issues that make them very different of the classical ones. 
To define the spaces  $L^{p(\cdot)}$, we consider a measurable function $\vf:\R^n\longrightarrow \R$, and the \emph{modular function} $m_{p(\cdot)}$ associated to $p(\cdot)$, which is defined by 
\begin{equation}\label{Def_Modular_Intro}
m_{p(\cdot)}(\vf)=\int_{\R^n}|\vf(x)|^{p(x)}dx.
\end{equation}
At this stage of the construction, is illustrative to see that, if we consider  $p(\cdot)\equiv p\in [1,+\infty[$, we can define the classical $L^p$-norm, and thus the $L^p$-space, by considering the expression 
$$\|\vf\|_{L^{p}(\R^n)}=\left(
 m_{p}(\vf)
 \right)^{\frac1p}.$$
In this point appears an important issue; if we consider a measurable function $p(\cdot)$ non constant, 
the previous formula does not have sense (because the exponent outside the integral). 
Then, the classical strategy to avoid this difficulty is to equip to the space $L^{p(\cdot)}$
with the following \emph{Luxemburg norm}:  
\begin{equation}\label{Def_LuxNormLebesgue}
\|\vf\|_{L^{p(\cdot)} (\R^n) }=\inf\left\{\lambda > 0: \, 
m_{p(\cdot)}\left( \frac{f}{\lambda} \right)
=
\int_{\R^n}\left|\frac{\vf(x)}{\lambda}  \right|^{p(x)}dx 
\leq1 \right\}.
\end{equation}
Thus, we define the variable Lebesgue spaces
$L^{p(\cdot)}(\R^n)$, as the set of measurable functions such that  quantity  $\|\cdot\|_{L^{p(\cdot)}} (\R^n)$ is finite.
For a complete presentation of the theory of variable Lebesgue spaces we refer to the interested reader to  the books  \cite{Cruz_Libro},  \cite{Diening_Libro} and \cite{BookALEX}. \\

With this information at hand, our objective in this paper is to provide a first application of the variable Lebesgue spaces to the analysis of the generalized semi-linear  power dissipative equation and the generalized convection-diffusion equation. For doing this, we present an unified approach which deals both cases. More precisely, we will construct solutions for the Cauchy problem \eqref{NS_Intro} provided with initial data and external forces in appropriate variable Lebesgue spaces. The presentation of our main results motivate the next subsection.

\subsection{Presentation of the results} 

In this subsection we state our main results about the existence and uniqueness of mild solution of the Cauchy problem for the generalized nonlinear heat equation \eqref{NS_Intro} on  variable Lebesgue spaces. 
Thus, in order to distinguish appropriately the different cases involved in \eqref{nonlinearity}, in the following we introduce some  useful notations. \\

\noindent Given a constant $s\in\R$, we denote by $    \langle
s
\rangle_\gamma$ to the quantity given by 
\begin{equation}\label{notation1}
    \langle
s
\rangle_\gamma 
 =
 \begin{cases}
0
   & \text{if } \gamma=0,
    \\[7pt]
  \displaystyle 
s
   & \text{if } \gamma=1.
\end{cases}
\end{equation}
Similarly, given the operators $\text{Id}(\cdot)$ and $ \vec{1}\cdot \vn (\cdot)$, 
we denote by $\vn^{\gamma}(\cdot)$ to the operator defined by the expression 
\begin{equation}\label{notation2}
\vn^{\gamma}(\cdot)
 =
 \begin{cases}
\text{Id}(\cdot)
   & \text{if } \gamma=0,
    \\[7pt]
  \displaystyle 
  \vec{1}\cdot
\vn (\cdot)
   & \text{if } \gamma=1.
\end{cases}
\end{equation}

With these notations fixed, let us proceed with the presentation of our first main result, which state 
the existence and uniqueness of a global mild solution of the Cauchy problem for the equations \eqref{NS_Intro} 
in the framework of variable Lebesgue spaces
as long the external force and the initial data  are sufficiently small.  This first theorem reads as follows. 

\begin{Theoreme}\label{Theoreme_2}
Let $\gamma \in \{0,1\}$ fixed. 
Consider $\alpha\in ]\frac 1 2, 1]$, a variable exponent $p(\cdot)\in \mathcal{P}^{\log}(\R^n)$ such that $p^->1$, an initial data $u_0\in \mathcal{L}^{p(\cdot)}_{   \frac{nb}{2\alpha - 
\langle 1 \rangle_\gamma
}      }(\R^n)$,
and let $f$ be an external force such that $\vf= \vn^{\gamma} ( \mathcal{F})$ where $\mathcal{F}$ is a function in  $ \mathcal{L}^{\frac{p(\cdot)}{b+1}}_{\frac{nb}{(b+1)(2\alpha
-
\langle 1 \rangle_\gamma
)}}(\R^n, L^\infty([0,T[ ))$.
If  $\|u_0\|_{\mathcal{L}^{p(\cdot)}_{   \frac{nb}{2\alpha -
\langle 1 \rangle_\gamma
}        }} +
\|\mathcal{F}\|_{\mathcal{L}^{\frac{p(\cdot)}{b+1}}_{\frac{nb}{(b+1)(2\alpha-
\langle 1 \rangle_\gamma
)},x}(L^\infty_t)}$ is small enough, then the equation (\ref{NS_Intro}) admits a unique global mild solution in the space 
$\mathcal{L}^{p(\cdot)}_{   \frac{nb}{2\alpha -
\langle 1 \rangle_\gamma
}      }
(\mathbb{R}^n,L^\infty([0,T[ ))$.
\end{Theoreme}
We must remark  the fact that,  we have studied the behaviour of the mild solution in the time variable by considering the $L^\infty $ space, and thus we have analyzed the information in the space variable  by considering the \emph{mixed variable Lebesgue spaces} $\mathcal{L}^{p(\cdot)}_{r}$ (see Subsection \ref{Secc_Notaciones_Presentaciones} for the definition of these spaces). 
Now, note that the mixed variable Lebesgue spaces considered here by merely technical reasons, and it is motivated by the lack of flexibility in the parameters that intervene in the boundedness of the Riesz transforms involved. In Subsection \ref{Secc_Notaciones_Presentaciones} and Remark \ref{Rem_Riesz_MixedLebesgue} below we will provide precise details on this particular issue. \\

In our second main result we will state the existence and uniqueness of a local mild solution  
for the equations \eqref{NS_Intro}  by considering a Lebesgue space of variable exponent in the time variable, and by setting a classical $L^q$-space in the space variable. This theorem reads as follows. 

\begin{Theoreme}\label{Theoreme_1}
Let $\gamma \in \{0,1\}$ fixed. Consider $\alpha\in ]\frac 1 2, 1],\ p(\cdot )\in \mPl(\Rt)$ with $b+1<p^-\leq p^+ <+\infty$, fix a parameter  
$q>\frac{nb}{2\alpha -  \langle1\rangle_\gamma  }$ by the relationship
$\frac{\alpha b}{p(\cdot)}
+
\frac{nb}{2q}
<\alpha - \langle \frac{1}{2} \rangle_\gamma$ and $\overline{q}(\cdot)\in\mP^{\text{emb}}_q(\Rt)$.
If $\vf\in L^{1} \left( [0,+\infty[,  L^{ \overline{q}(\cdot)  } (\R^n) \right)$ is an external force and if the initial data 
$u_0\in L^{ \overline{q}(\cdot)  } (\Rt)$, 
then there exist a time $0<T<+\infty$ and a unique mild solution of the equation \eqref{NS_Intro} in the space $L^{p(\cdot)} \left( [0,T], L^{q} (\R^n) \right)$.
\end{Theoreme}

Note that, the variable exponent $\overline{q}(\cdot)$ involved in this theorem belongs to the class $\mP^{\text{emb}}_q(\Rt)$, which characterize the variable Lebesgue spaces embedded in the space $L^q(\Rt)$. We will give a precise notion of this class of exponents in Definition \ref{class of embedding} and Lemma  \ref{lema embeddd set} below.\\ 

To finish this subsection we must mention that the theorems recently presented here seems to be, to the best of our knowledge, are the first applications  of the variable Lebesgue spaces in the analysis of the generalized nonlinear heat equations  \eqref{NS_Intro}. We hope that these results will inform future studies about applications of variable Lebesgue spaces to various evolution PDEs. 

\subsection*{Organization of the paper}
The next subsections are structured as follows. In Section \ref{Secc_resentaciones} we will present a review of the main definitions and properties of  Variable Lebesgue spaces $L^{p(\cdot)}$ and  some decay estimates of fractional heat kernels. Section \ref{Secc_Proof_Existence} is dedicated to the proof of  Theorems \ref{Theoreme_2} and \ref{Theoreme_1}.

\section{Preliminaries}\label{Secc_resentaciones}
With the aim of keeping this article reasonably self-contained,
in this section we will present the key results and definitions involved in the proofs of our main results. 
To start, we present the following classical result which will be the core idea to obtain our mild solutions: {\color{black} 
\begin{Proposition}[Contraction mapping principle]\label{cmp}
Let $X$ be a Banach space, let $Z \subset X$ be a closed bounded subset, and let $\Phi : Z \rightarrow Z$ be Lipschitz continuous in the norm topology with Lipschitz constant $L < 1$, i.e.  given $x,y\in Z$, $\|\Phi(x) - \Phi(y)\| \leq L \|x - y\|$, with $L < 1.$ Then, $\Phi$ has a unique fixed point in $Z$.
\end{Proposition} }
\noindent The interested reader can consult the Appendix of the book \cite{bedrossian2022mathematical} for a proof of this proposition.

\medskip

To continue, in the next subsections we present some tools on variable Lebesgue spaces and fractional heat kernels.

\subsection{Variable Lebesgue spaces}\label{Secc_Notaciones_Presentaciones}
We begin this subsection by specifying some of the basic concept involved in the theory of 
variable Lebesgue spaces;
the notions of variable exponent,  set of variable exponents and limit exponents.


\begin{Definition}
    Let  $n\in \mathbb{Z}_{+}$  and  $p:\mathbb{R}^n\longrightarrow [1,+\infty[$ a function.
    We say that  $p(\cdot)$ is a {\it variable exponent} if it is a measurable function. We denote by $\mathcal{P}(\mathbb{R}^n)$ the set of variable exponents, and we define the {\it limit exponents} $p^-$ and $p^+$ as 
    \begin{equation}
        p^-= {\mbox{inf ess}}_{x\in \mathbb{R}^n}  \; \{p(x)\}
        \quad 
        \text{and}
        \quad 
        p^+= {\mbox{sup ess}}_{x\in \mathbb{R}^n} \; \{p(x)\}.
    \end{equation}
\end{Definition}
Given a variable exponent $p(\cdot)$, in the rest of the article we will consider that
$$1<p^-\leq p^+<+\infty.$$ 
In this point we must emphasise the fact that the spaces $L^{p(\cdot)}(\mathbb{R}^n)$  are  Banach function spaces, and thus they possess interesting  and natural properties. 
In the next, we present the   generalization of the H\"older inequalities to this functional setting.
\begin{Lemme}
    Let consider the variable exponents $p_1(\cdot),\,p_2(\cdot),\,p_3(\cdot)\in \mathcal{P}(\mathbb{R}^n)$ such that the following pointwise relationship follows
$\frac{1}{p_1(x)}=\frac{1}{p_2(x)}+\frac{1}{p_3(x)}$, $x\in \mathbb{R}^n$. Then, given $f\in L^{p_2(\cdot)}(\mathbb{R}^n) $ and $g \in L^{p_3(\cdot)}(\mathbb{R}^n)$, the pointwise product $fg$ belongs to the space $L^{p_1(\cdot)}(\mathbb{R}^n)$, and there exists a numerical constant $C>0$ such that 
\begin{equation}\label{Holder_LebesgueVar}
\|f g\|_{L^{p_1(\cdot)}} \leq C\|f\|_{L^{p_2(\cdot)}}\|g\|_{L^{p_3(\cdot)}}.
 \end{equation}
\end{Lemme}
Note that, given two vector fields $\vec{f}, \vg :\mathbb{R}^n\longrightarrow \mathbb{R}^n$, the estimate \eqref{Holder_LebesgueVar} can be  generalized to the product $\vec{f}\cdot \vg$. 
The proof of this result can be consulted in  \cite[Section 2.4]{Cruz_Libro} or \cite[Section 3.2]{Diening_Libro}.\\

To continue we remark that
the Luxemburg norm $\|\cdot\|_{L^{p(\cdot)}}$ defined in \eqref{Def_LuxNormLebesgue} satisfies the \emph{norm conjugate formula} \cite[Corollary 3.2.14]{Diening_Libro}. This result reads as follow.

\begin{Proposition}
    Let consider two variable exponents $p_1(\cdot),p_2(\cdot)\in \mathcal{P}(\mathbb{R}^n)$ such that 
$1=\frac{1}{p_1(x)}+\frac{1}{p_2(x)}$, for all $x\in \mathbb{R}^n$. 
Then, given   $f\in L^{p_1(\cdot)}$, there exists a numerical constant $C>0$ such that
\begin{equation}\label{Norm_conjugate_formula}
\|f\|_{L^{p_1(\cdot)}}
\leq
2 \underset{\|g\|_{L^{p_2(\cdot)}}\leq 1}{\sup}\int_{\mathbb{R}^n}|f(x)||g(x)|dx.
\end{equation}
\end{Proposition}

\begin{Remarque}
As is natural the notions and results recently presented in the setting of 
the whole space $\mathbb{R}^n$ can be adapted to any subset $\Omega \subset\mathbb{R}^n$. 
\end{Remarque}
In the next we present some embedding results in the setting of variable Lebesgue spaces. 
\begin{Lemme}\label{lemme_embeding}
Let  $n\geq 1$, a bounded domain $\Omega\subset \mathbb{R}^n$ and   $p_1(\cdot),  p_2(\cdot)\in \mathcal{P}(\Omega)$ such that $1< p_1^+, \ p_2^+ <+\infty$.  Then, $L^{p_2(\cdot)}(\Omega)\subset L^{p_1(\cdot)}(\Omega)$ if and only if $p_1(x)\leq p_2(x)$ almost everywhere. Morevover, in the case we have the estimate
$$\|f\|_{L^{p_1(\cdot)} (\Omega) } \leq\left(1+\left|  \Omega\right|\right)\|f\|_{L^{p_2(\cdot)} (\Omega)  }.$$
\end{Lemme}
The proof of this lemma can be consulted in \cite[Corollary 2.48]{Cruz_Libro}. \\

An interesting fact in the setting of variable Lebesgue spaces is given by the extension of the result presented in Lemma \ref{lemme_embeding} to unbounded domains. 
In particular, we will be interested in the case when $p_1(x)$ is a constant function. In order to characterize it in a concise manner  we introduce the following class of exponents. 
\begin{Definition}\label{class of embedding}
Given a constant exponent 
$q\in (1,\infty)$, we define the class of variable exponents $\mP^{\text{emb}}_q(\Rt)$, as the set
    \begin{equation}
        \mP^{\text{emb}}_q(\Rt) =
        \left\{
\overline{q}(\cdot)\in \mPl(\Rt) : 
q\leq(\overline{q})^- \leq (\overline{q})^+ <+\infty
 \quad \text{ and } \quad 
\frac{ q \overline{q}(x)}{ \overline{q}(x) - q } \to +\infty 
\text{ as } |x|\to +\infty
        \right\}.
    \end{equation}
\end{Definition}
A nice and natural consequence of considering a variable exponent in the class $\mP^{\text{emb}}_q(\Rt)$
is given in the 
following result, 
which follows
by considering Theorem 2.45 and Remark 2.46 in \cite{Cruz_Libro}. 
\begin{Lemme}\label{lema embeddd set}
    Let $q\in (1,\infty)$ and  $\overline{q}(\cdot)\in\mP^{\text{emb}}_q(\Rt)$. Then,  
    $L^{\overline{q} (\cdot)}(\Rt)\subset L^{q}(\Rt)$, and there exists a numerical constant $C>0$ such that 
    $$\|f\|_{L^{q} (\Rt)} \leq C\|f\|_{L^{\overline{q}(\cdot)} (\Rt)  }.$$
\end{Lemme}

In this point we must emphasize that
the variable Lebesgue spaces  are not translation invariant, and thus, the Young's inequality for convolution cannot be generalized to  $L^{p(\cdot)}$ spaces for non-constant exponents (see \cite[Section 3.6]{Diening_Libro}). 
In consequence, new ideas are needed to deal with 
many of the classical operators and techniques that appear in the analysis of PDEs. \\

One of the classical techniques to study the boundedness  of such operators on $L^{{p}(\cdot)}$ spaces, 
is  to  consider some constrains on the variable exponents. The most important condition on these exponents is given by the {\it log-H\"older continuity condition}, which we present in the following.

\begin{Definition}
Let consider a variable exponent $p(\cdot)\in \mathcal{P}(\mathbb{R}^n)$ such that there exists  the limit value 
$
\frac{1}{p_\infty}=\underset{|x|\to +\infty}{\lim}\frac{1}{p(x)}.$ 
\begin{enumerate}
    \item We say that $p(\cdot)$ is locally log-H\"older continuous if  for each $x,y\in \mathbb{R}^n,$ there exists a constant $C>0$ such that 
    $$\left|\frac{1}{p(x)}-\frac{1}{p(y)}\right|\leq \frac{C}{\log(e+1/|x-y|)}.$$ 
    \item We say that $p(\cdot)$ satisfies the log-H\"older decay condition,  if  for each $x\in \mathbb{R}^n,$ there exists a constant $C>0$ such that 
    $$\left|\frac{1}{p(x)}-\frac{1}{p_\infty}\right| \leq \frac{C}{\log(e+|x|)}.$$
    \item We say that $p(\cdot)$ is globally log-H\"older continuous in $\mathbb{R}^n$ if it is locally  log-H\"older continuous and satisfies the log-H\"older decay condition.
    \item We define the class as the following class of variable exponents
    \begin{equation}
        \mathcal{P}^{log}(\mathbb{R}^n) =
        \left\{
p(\cdot)\in \mP(\Rt) : 
p(\cdot) \text{ is  globally log-H\"older continuous in }\mathbb{R}^n
        \right\}.
    \end{equation}
\end{enumerate}
\end{Definition}
To continue,  we recall the definition of the Hardy-Littlewood maximal function.

\begin{Definition}
Let $f:\mathbb{R}^n\longrightarrow \mathbb{R}$ a locally integrable function. The Hardy-Littlewood maximal function $\mathcal{M}$ is defined by 
$$\displaystyle{\mathcal{M}(f)(x)=\underset{B \ni x}{\sup } \;\frac{1}{|B|}\int_{B }|f(y)|dy}$$ where $B$ is an open ball of $\mathbb{R}^n$.
\end{Definition}
With the notions of globally log-H\"older continuous exponents  we can present the following key result 
(see \cite[Section 4.3]{Diening_Libro}). 
\begin{Theoreme}\label{MaximalFunc_LebesgueVar}
    Let  consider a variable exponent $p(\cdot)\in \mathcal{P}^{log}(\mathbb{R}^n)$ satisfying  $p^->1$. Then, given $f\in L^{p(\cdot)}$, there exist a constant $C>0$ such that
\begin{equation}
\|\mathcal{M}(f)\|_{L^{p(\cdot)}}\leq C \|f\|_{L^{p(\cdot)}}.
\end{equation}
\end{Theoreme}
In the next  we recall a classical Lemma about the Hardy-Littlewood maximal function (see \cite[Section 2.1]{grafakos2008classical}):
\begin{Lemme}\label{lemme_conv_maximal}
If $\varphi$ is a radially decreasing function on $\mathbb{R}^n$ and $\vf$ is a locally integrable function, then 
\begin{equation*}
|(\varphi\ast\vf)(x)| \leq \|\varphi\|_{L^1}  \mathcal{M} (\vf)(x),
\end{equation*}
where $\mathcal{M}$ is the Hardy-Littlewood maximal function.
\end{Lemme}
Let consider the Riesz transforms $(\mathcal{R}_j)_{1\leq j\leq n}$, which are defined  at the Fourier level by the expression $$\widehat{\mathcal{R}_j(f)}(\xi)=-\frac{i\xi_j}{|\xi|}\widehat{f}(\xi).$$ 
The next result state that the Riesz transform of a measurable function 
are also bounded in Lebesgue spaces of variable exponent.
A proof of this fact can be consulted in \cite[Sections 6.3 and 12.4]{Diening_Libro}.
\begin{Lemme}
Let consider a variable exponent  $p(\cdot)\in \mathcal{P}^{log}(\mathbb{R}^n)$ satisfying $1<p^-\leq p^+<+\infty$. Then, given $f\in L^{p(\cdot)}$, there exist a constant $C>0$ such that
\begin{equation}\label{Riesz_LebesgueVar}
\|\mathcal{R}_j(f)\|_{L^{p(\cdot)}}\leq C \|f\|_{L^{p(\cdot)}},
\end{equation}
\end{Lemme}
In the following we  recall the notion of  Riesz potential operator.
\begin{Definition}\label{Definition_RieszPotential}
Let consider a parameter $\beta \in (0,n)$ and  a measurable function $f$.  We define the Riesz potential operator $\mathcal{I}_\beta (f):\R^n \to [0,+\infty]$  by the expression 
\begin{equation}
\mathcal{I}_\beta(f)(x):=\int_{\mathbb{R}^n}\frac{|f(y)|}{|x-y|^{n-\beta}}dy. 
\end{equation}
\end{Definition}
This operator is bounded in variable Lebesgue spaces if we consider appropriate globally log-H\"older continuous variable exponents. 
A precise statement of this important is presented in the next.
\begin{Theoreme}\label{theo.PotentialRieszVariable0}
Let consider a variable exponent  $p(\cdot)\in \mathcal{P}^{log}(\mathbb{R}^n)$ and a parameter $\beta\in (0,n/p^+)$. Then, there exist a numerical constant $C>0$ such that the following estimate follows
\begin{equation}\label{PotentialRieszVariable0}
\|\mathcal{I}_\beta(f)\|_{L^{q(\cdot)}}\leq C\|f\|_{L^{p(\cdot)}},\qquad \mbox{with} \quad  \frac{1}{q(\cdot)}=\frac{1}{p(\cdot)}-\frac{\beta}{n}.
\end{equation}
\end{Theoreme}
A proof of this result  can be consulted in 
\cite[Section 6.1]{Diening_Libro}.
\begin{Remarque}
    Note that the estimate \eqref{PotentialRieszVariable0} introduces a very strong relationship between the variable exponents $p(\cdot)$ and $q(\cdot)$. For instance, if $p(\cdot)$ is a constant function, this forces  to the exponent $q(\cdot)$ to be constant function as well. 
\end{Remarque}
In order  to provide   more flexibility on these parameters (see Remark \ref{Rem_Riesz_MixedLebesgue} below) we will consider the following functional spaces (see  in \cite{chamorro_paper_lpvar} for more details).
\begin{Definition}
Let consider a variable exponent  $p(\cdot)\in \mathcal{P}^{\log}(\mathbb{R}^n)$  and $q\in (1,+\infty)$. Then, we define  the mixed Lebesgue space $\mathcal{L}^{p(\cdot)}_q(\mathbb{R}^n)$ as   $$\mathcal{L}^{p(\cdot)}_q(\mathbb{R}^n)=L^{p(\cdot)}(\mathbb{R}^n)\cap L^{q}(\mathbb{R}^n),$$ which can be normed 
by the expression 
\begin{equation}\label{MixedLebesgue}
\|\cdot\|_{\mathcal{L}^{p(\cdot)}_q}=\max\{\|\cdot\|_{L^{p(\cdot)}}, \|\cdot\|_{L^{q}}\}.
\end{equation}
\end{Definition}
As we mentioned above, provided with these functional spaces we have the following result (see \cite{chamorro_paper_lpvar}). 
\begin{Proposition}\label{Proposition_RieszPotential}
Let consider $q\in (1,+\infty)$, a variable exponent $p(\cdot)\in \mathcal{P}^{\log}(\mathbb{R}^n)$ and a parameter $0<\beta<\min\{n/p^+, n/q\}$. Given a function $f\in \mathcal{L}^{p(\cdot)}_q(\mathbb{R}^n)$ and a variable exponent  $r(\cdot)$ defined by the following condition
\begin{equation}\label{Inegalite2Condition}
r(\cdot)=\frac{np(\cdot)}{n-\beta q},
\end{equation} then, there exist a numerical constant $C>0$  such that the following estimate follows
\begin{equation}\label{Inegalite3}
\|\mathcal{I}_\beta(f)\|_{L^{r(\cdot)}}\leq C \|f\|_{\mathcal{L}^{p(\cdot)}_q}.
\end{equation}
\end{Proposition}

\begin{Remarque} 
Note that the index $q$ is not to related to the limit exponents $p^-$ or $p^+$ nor to $p(\cdot)$. 
\end{Remarque}

\begin{Remarque}\label{Rem_Holder_Mixed_Lebesgue_Var}
Note that the functional spaces $\mathcal{L}^{p(\cdot)}_q$ inherit the properties of the spaces $L^{p(\cdot)}$ and $L^{q}$.
Thus, given variable exponents $p_1(\cdot), p_2(\cdot), p_3(\cdot), $ and constant exponents 
$q, q_2, q_3, $ such that 
$\frac{1}{p_1(\cdot)}=\frac{1}{p_2(\cdot)}+\frac{1}{p_3(\cdot)}$ and $\frac{1}{q_1}=\frac{1}{q_2}+\frac{1}{q_3}$
the following H\"older-type inequality follows $$\| fg \|_{\mathcal{L}^{p_1(\cdot)}_{q_1}}\leq \|f\|_{\mathcal{L}^{p_2(\cdot)}_{q_2}}
\|g\|_{\mathcal{L}^{p_3(\cdot)}_{q_3}}$$ Furthermore,  the Riesz transforms are also bounded in these spaces.\\
\end{Remarque}
For more details about the theory of Variable Lebesgue spaces, we recommend to the interested reader to see the books \cite{Cruz_Libro},  \cite{Diening_Libro} and \cite{BookALEX}.\\

\subsection{Some key estimates on the fractional heat kernel}
In this subsection we state some estimates related to the fractional heat kernel $\mathfrak{g}_t^\alpha (x)$ involved in the mild formulation of the system  \eqref{NS_Intro}.  Let us begin by recalling that this kernel is defined by
\begin{equation*}
\mathfrak{g}_t^\alpha (x)
=
\mathcal{F}^{-1}(e^{-t|\xi|^{2\alpha}})(x)
=
(2\pi)^{-\frac{n}{2}}\int_{\mathbb{R}^{n}}e^{ix\cdot\xi}e^{-t|\xi|^{2\alpha}}d\xi.
\end{equation*}

\begin{Remarque}\label{remark 2.1 MIAO}
Let consider $x\in \R^n$ and the fractional heat kernel $\mathfrak{g}_t^\alpha (x)$.
Then, there exists a numerical constant $C>0$ such that the following pointwise estimates follow 
\begin{equation}\label{eq 1 1 24}
|\mathfrak{g}_t^\alpha (x)| 
\leq
C
\dfrac{t}{(t^{\frac{1}{2\alpha}}+|x|)^{n+2\alpha}}
\quad \text{ and }\quad
|\nabla  \mathfrak{g}_t^\alpha (x)| 
\leq
C
\dfrac{1}{(t^{\frac{1}{2\alpha}}+|x|)^{n+1}}
.
\end{equation}
\end{Remarque}
For a precise statement of this remark see  \cite[Lemma 2.1  and Remark 2.1]{miao2008well}.

\begin{Lemme}\label{Lemma 3.1 MIAO}
Let consider  $x\in \R^n$, the fractional heat kernel $\mathfrak{g}_t^\alpha (x)$,  and  the parameters $1\leq p \leq q \leq +\infty $. Then, given $\alpha,\nu >0$,   there exists a numerical constant $C>0$ such that following estimates follow
\begin{enumerate}
    \item  $\left\|
    \mathfrak{g}_t^\alpha *
    f(x)\right\|_{L^q}
\leq
C t^{-\frac{n}{2 \alpha}\left(\frac{1}{p}-\frac{1}{q}\right)}\|f\|_{L^p}$,
\item $\left\|(-\Delta)^{\frac{\nu}{2}} \mathfrak{g}_t^\alpha * 
f(x)\right\|_{L^q} \leq C t^{-\frac{v}{2 \alpha}-\frac{n}{2 \alpha}\left(\frac{1}{p}-\frac{1}{q}\right)}\|f\|_{L^p}$.
\end{enumerate}
\end{Lemme}

We refer to the interested reader to \cite[Lemma 3.1]{miao2008well} for a proof of these results. 
\section{Mild solutions in variable Lebesgue spaces}\label{Secc_Proof_Existence}

In this section we study mild solutions for the nonlinear heat equation (\ref{NS_Intro}) in the setting of variable Lebesgue spaces. These mild solutions are obtained via the classical contraction mapping principle.
Note that, by considering notations introduced in \eqref{notation1} and \eqref{notation2},  the equations stated in \eqref{NS_Intro} can be recast as
\begin{equation}\label{NS_Intro2}
\begin{cases}
\partial_t u=- (-\Delta)^{\alpha} u
+
\vn^{\gamma} (| u|^b u )
+  f, &   (t,x)\in ]0,+\infty[\times \R^n, \ 
\\[7pt]
u(0,x)=u_0(x), & x\in \mathbb{R}^n,
\end{cases}
\end{equation}
where $\gamma \in \{0,1\}$. Now, due to the Duhamel formula, we can recover  equation (\ref{NS_Intro2})  in the following equivalent form 
\begin{equation}\label{NS_Integral}
u(t,x)=
\mathfrak{g}^\alpha_t\ast u_0(x)
+
\int_{0}^t
\mathfrak{g}^\alpha_{t-s}\ast\vf(s, x)
ds
-
\int_{0}^t
\mathfrak{g}^\alpha_{t-s}\ast 
\vn^{\gamma} (|u|^b u )
(s, x)
ds,
\end{equation}
where $\mathfrak{g}^\alpha_t$ denotes the fractional heat kernel. 
Any function satisfying the integral equation \eqref{NS_Integral}  is called a \textit{mild solution} of \eqref{NS_Intro2} (or \eqref{NS_Intro}).
In the next subsections we will apply the contraction mapping principle on this integral equation. 

\subsection{Proof of  Theorem \ref{Theoreme_2}}
In this subsection we consider a variable exponent $p(\cdot)\in \mathcal{P}^{log}(\R^+)$ such that $1<p^-\leq 
p(\cdot)
\leq p^+<+\infty$, and the functional space 
$$
\mathcal{E}=
\mathcal{L}^{p(\cdot)}_{   \frac{nb}{2\alpha -
\langle 1 \rangle_\gamma
}      }
(\R^n,L^\infty([0,T[ )),$$  
which is dotted with the following norm
\begin{equation}\label{Norm_PointFixe}
\|\cdot\|_{\mathcal{E}}=\max\left\{\|\cdot\|_{L^{p(\cdot)}_x(L^\infty_t)}, \|\cdot\|_{
L^{   \frac{nb}{2\alpha -
\langle 1 \rangle_\gamma
}        }_x(L^\infty_t)}
\right\}.
\end{equation}
As mentioned above, our strategy will be construct mild solutions for the integral equation  (\ref{NS_Integral}) by considering 
mapping contraction principle on thus functional setting. More precisely, the proof of Theorem \ref{Theoreme_2} is based in the next 3 propositions.
\\

Let us begin by considering the following result regarding a control of the initial data.

\begin{Proposition}\label{prop.Control_Uo}
Let $\gamma \in \{0,1\}$ fixed. Consider $\alpha\in ]\frac 1 2, 1]$,  a variable exponent $p(\cdot)\in \mathcal{P}^{log}(\R^n)$ such that $p^->1$, and 
 a  function $u_0\in \mathcal{L}^{p(\cdot)}_{   \frac{nb}{2\alpha -\langle 1 \rangle_\gamma}      }(\R^n)$. Then, there exist a numerical constant $C>0$ such that 
 \begin{equation}\label{Control_Uo}
\|\mathfrak{g}_t^\alpha\ast u_0\|_{\mathcal{E}}\leq C\|u_0\|_{\mathcal{L}^{p(\cdot)}_{   \frac{nb}{2\alpha -
\langle 1 \rangle_\gamma
}        } }.
\end{equation}
\end{Proposition}
\begin{proof}
Note that $u_0$ is a locally integrable function, since $u_0\in L^{\frac{nb}{2\alpha -
\langle 1 \rangle_\gamma
}}\subset L^1_{\text{loc}}$. Then, as the  fractional heat kernel $\mathfrak{g}_t^\alpha$ is a  radially decreasing function, 
we can consider   Lemma \ref{lemme_conv_maximal} to obtain 
$$\|\mathfrak{g}_t^\alpha\ast u_0(x)\|_{L^\infty_t}\leq C\mathcal{M}(u_0)(x).$$
Recalling  
the norm defined in \eqref{Norm_PointFixe}, we can write the estimate
\begin{eqnarray*}
\|\mathfrak{g}_t^\alpha\ast u_0\|_{\mathcal{E}}
&\leq & C\max\left\{\|\mathcal{M}(u_0)\|_{L^{p(\cdot)}}, \|\mathcal{M}(u_0)\|_{L^{\frac{nb}{2\alpha -
\langle 1 \rangle_\gamma
}  }}
\right\}.
\end{eqnarray*}
Now, by hypothesis we know that
$p(\cdot)\in \mathcal{P}^{log}(\R^n)$ with $p^->1$, then, by Theorem \ref{MaximalFunc_LebesgueVar} we have that the maximal function $\mathcal{M}$ is bounded on the Lebesgue space $L^{p(\cdot)}(\R^n)$.
Considering this fact, and since  $\mathcal{M}$ is also bounded on the space $L^{\frac{nb}{2\alpha -
\langle 1 \rangle_\gamma
}  }$,  we conclude
$$\|\mathfrak{g}_t^\alpha\ast u_0\|_{\mathcal{E}}\leq C\max\big\{\|u_0\|_{L^{p(\cdot)}},\|u_0\|_{L^{\frac{nb}{2\alpha -
\langle 1 \rangle_\gamma
}  }}\big\}\leq  C\|u_0\|_{\mathcal{L}^{p(\cdot)}_{\frac{nb}{2\alpha -
\langle 1 \rangle_\gamma
}  }}.$$
With this we conclude the proof. 
\end{proof}

\begin{Proposition}\label{prop.Control_f_preambulo}
Let $\gamma \in \{0,1\}$ fixed.
Consider $\alpha\in ]\frac 1 2, 1]$, a variable exponent $p(\cdot)\in \mathcal{P}^{\log}(\R^n)$ with $p^->1$,  and a function $f=\vn^{\gamma}(\mathcal{F})$, where $\mathcal{F}\in \mathcal{L}^{\frac{p(\cdot)}{b+1}}_{\frac{nb}{(b+1)(2\alpha-\langle 1 \rangle_\gamma)}}(\R^n, L^\infty([0,T[ ))$.
Then, there exist a numerical constant $C>0$ such that 
\begin{equation}\label{Control_f_preambulo}
\left\|\int_{0}^t\mathfrak{g}_{t-s}^\alpha\ast 
f (\cdot, \cdot)ds\right\|_{\mathcal{E}}
\leq C
\|\mathcal{F}\|_{
\mathcal{L}^{\frac{p(\cdot)}{b+1}}_{\frac{nb}{(b+1)(2\alpha-
\langle 1 \rangle_\gamma
)},x}(L^\infty_t)
}. 
\end{equation}
\end{Proposition}
\begin{proof}  
To conclude the estimate \eqref{Control_f_preambulo},  in the following we will study separately the cases when $\gamma=1$ and $\gamma=0$. 
\begin{itemize}
\item {\bf Case $\gamma=1$}.
We begin the analysis of this case by noticing that
since  $f=\vn^{\gamma}(\mathcal{F})$ we can write
$$\left|\int_{0}^t\mathfrak{g}_{t-s}^\alpha\ast 
 f
(s,x)ds\right|\leq C\int_{0}^t\int_{\R^n}|\vn\mathfrak{g}_{t-s}^\alpha(x-y)| |\mathcal{F}(s,y)|dyds.$$
Then, by Fubini's theorem
and the decay properties of the fractional heat kernel in Remark \ref{remark 2.1 MIAO}, we get 
\begin{eqnarray*}
\left|\int_{0}^t\mathfrak{g}_{t-s}^\alpha\ast  f
(s,x)ds\right|
&\leq &C\int_{\R^n}\int_{0}^t
\frac{1}{  
 (\ |t-s|^{ \frac{1}{
       2\alpha }   } +
|x-y| \ )^{n+1}
}
|\mathcal{F}(s,y)|dsdy.
\end{eqnarray*}
Now, by considering the $L^\infty_t$ norm on $\mathcal{F}$, 
we ob
$$\left|\int_{0}^t\mathfrak{g}_{t-s}^\alpha\ast 
 f
(s,x)ds\right|
\leq C\int_{\R^n}\int_{0}^t
\frac{1}{  
 (\ |t-s|^{ \frac{1}{
       2\alpha }   } +
|x-y| \ )^{n+1}
}
ds \|\mathcal{F}(\cdot,y)\|_{L^\infty_t}dy.$$
Then, as  $\alpha\in ]\frac 1 2, 1]$, we can consider the Riesz potential (see definition (\ref{Definition_RieszPotential}))
provided of the estimate
\begin{equation}\label{est 3.5}
\begin{aligned}
\int_0^t \frac{d s}{\left(|t-s|^{\frac{1}{2 \alpha}}+|x-y|\right)^{n+1}} & \leq \int_0^{+\infty} \frac{d s}{\left(s^{\frac{1}{2 \alpha}}+|x-y|\right)^{n+1}} \\
& =\int_0^{+\infty} \frac{|x-y|^{2 \alpha} d \beta}{\left(\left(|x-y|^{2 \alpha} \beta\right)^{\frac{1}{2 \alpha}}+|x-y|\right)^{n+1}}  
\\ & =\frac{1}{|x-y|^{n+1-2 \alpha}} 
\int_0^{+\infty} \frac{d \beta}{\left(1+\beta^{\frac{1}{2 \alpha}}\right)^{n+1}},
\end{aligned}
\end{equation}
in order to obtain 
$$\left|\int_{0}^t\mathfrak{g}_{t-s}^\alpha\ast 
 f
(s,x)ds\right|\leq C\int_{\R^n}\frac{1}{|x-y|^{n+1-2\alpha}} \|\mathcal{F}(\cdot,y)\|_{L^\infty_t}dy=
C
\mathcal{I}_{2\alpha-1}(\|\mathcal{F}(\cdot,\cdot)\|_{L^\infty_t})(x),$$
and then, we can write 
\begin{equation}\label{eq.24.janv}
\left\|\int_{0}^t\mathfrak{g}_{t-s}^\alpha\ast  f(s,x)ds\right\|_{L^\infty_t}\leq C\mathcal{I}_{2\alpha-1}(\|\mathcal{F}(\cdot,\cdot)\|_{L^\infty_t})(x).
\end{equation}
Now, to reconstruct the $\mathcal{L}^{p(\cdot)}_{\frac{nb}{2\alpha-1}}$-norm given in (\ref{Norm_PointFixe}), from the estimate \eqref{eq.24.janv} we write 
\begin{eqnarray*}
\left\|\int_{0}^t\mathfrak{g}_{t-s}^\alpha\ast  f(s,x)ds\right\|_{L^{p(\cdot)}_x(L^\infty_t)}&\leq &C\left\|\mathcal{I}_{2\alpha-1}(\|\mathcal{F}(\cdot,\cdot)\|_{L^\infty_t})(\cdot)\right\|_{L^{p(\cdot)}_x},
\end{eqnarray*}
and 
\begin{eqnarray*}
\left\|\int_{0}^t\mathfrak{g}_{t-s}^\alpha\ast  f(s,x)ds\right\|_{L^{\frac{nb}{2\alpha-1}}_x(L^\infty_t)}&\leq &C\left\|\mathcal{I}_{2\alpha-1}(\|\mathcal{F}(\cdot,\cdot)\|_{L^\infty_t})(\cdot)\right\|_{L^{\frac{nb}{2\alpha-1}}_x}.
\end{eqnarray*}
Thus, by Proposition \ref{Proposition_RieszPotential} we get 
\begin{equation}\label{eq 2 24 janv}
\left\|\mathcal{I}_{2\alpha-1}(\|\mathcal{F}(\cdot,\cdot)\|_{L^\infty_t})(\cdot)\right\|_{L^{p(\cdot)}_x}\leq C\left\|\|\mathcal{F}(\cdot,\cdot)\|_{L^\infty_t}\right\|_{\mathcal{L}^{\frac{p(\cdot)}{b+1}}_{\frac{nb}{(b+1)(2\alpha-1)},x}}=\|\mathcal{F}\|_{\mathcal{L}^{\frac{p(\cdot)}{b+1}}_{\frac{nb}{(b+1)(2\alpha-1)},x}(L^\infty_t)}.
\end{equation}
Note that, since the Riesz potentials  are bounded on space $L^{\frac{nb}{2\alpha-1}}$,  we obtain 
\begin{equation}\label{eq 3 24 janv}
\left\|\mathcal{I}_{2\alpha-1}(\|\mathcal{F}(\cdot,\cdot)\|_{L^\infty_t})(\cdot)\right\|_{L^{\frac{nb}{2\alpha-1}}_x}\leq C\left\|\|\mathcal{F}\|_{L^\infty_t}\right\|_{L^{\frac{nb}{(b+1)(2\alpha-1)}}_x}=\|\mathcal{F}\|_{L^{\frac{nb}{(b+1)(2\alpha-1)}}_x(L^\infty_t)}.
\end{equation}
Then, 
gathering together 
the definition of the norm $\|\cdot\|_{\mathcal{E}}$ given in (\ref{Norm_PointFixe})
with the estimates 
\eqref{eq 2 24 janv} and
\eqref{eq 3 24 janv}, we conclude  
\begin{equation}\label{vase gamma 1 f}
   \left\|\int_{0}^t\mathfrak{g}_{t-s}^\alpha\ast  f(s,x)ds\right\|_{
\mathcal{L}^{ {p(\cdot)} }_{\frac{nb}{ (2\alpha-1)},x}(L^\infty_t)
   } 
\leq C
\|\mathcal{F}\|_{
\mathcal{L}^{\frac{p(\cdot)}{b+1}}_{\frac{nb}{(b+1)(2\alpha-1)},x}(L^\infty_t)
}
. 
\end{equation}

\item {\bf Case $\gamma=0$}.
Similarly to the previous case, if we consider Fubini's theorem
and the decay properties of the fractional heat kernel in Remark \ref{remark 2.1 MIAO}, we can write 
\begin{eqnarray*}
\left|
\int_{0}^t \mathfrak{g}^\alpha_{t-s}\ast   \mathcal{F}  (s, x) ds
\right|
&\leq &
C\int_{\mathbb{R}^n}\int_{0}^t
\frac{|t-s|}{   (\ |t-s|^{ \frac{1}{
       2\alpha }   } +
|x-y| \ )^{n+2\alpha}
}
 |\mathcal{F}(s,y)|dsdy,
\end{eqnarray*}
and then, involving the $L^\infty_t$-norm on the previous estimate,  we conclude  
\begin{eqnarray*}
\left|
\int_{0}^t
\mathfrak{g}^\alpha_{t-s}\ast    \mathcal{F} (s, x) 
ds
\right|
&\leq & 
C\int_{\R^n}
\left(\int_{0}^t 
\frac{|t-s|}{  
 (\ |t-s|^{ \frac{1}{
       2\alpha }   } +
|x-y| \ )^{n+2\alpha}
}
ds\right)
\|\mathcal{F}(\cdot,y)\|_{L^\infty_t}
dy.
\end{eqnarray*}
Now, considering  the estimate
\begin{equation}\label{est 3.10}
\begin{aligned}
\int_0^t \frac{ |t-s| ds}{\left(|t-s|^{\frac{1}{2 \alpha}}+|x-y|\right)^{n+2\alpha}} 
& 
\leq 
{\color{black}
\int_0^{+\infty} \frac{s d s}{\left(s^{\frac{1}{2 \alpha}}+|x-y|\right)^{n+2 \alpha}}
} \\
& =\int_0^{+\infty} \frac{|x-y|^{2 \alpha}\beta |x-y|^{2 \alpha}d \beta}{\left(\left(|x-y|^{2 \alpha} \beta\right)^{\frac{1}{2 \alpha}}+|x-y|\right)^{n+2\alpha}}  \\
&=\frac{1}{|x-y|^{n-2 \alpha}} 
\int_0^{+\infty} \frac{\beta d \beta}{\left(1+\beta^{\frac{1}{2 \alpha}}\right)^{n+2\alpha}},
\end{aligned}
\end{equation}
by the definition of the Riesz potential (see definition (\ref{Definition_RieszPotential})), we obtain 
\begin{eqnarray*}
\left|
\int_{0}^t
\mathfrak{g}^\alpha_{t-s}\ast    \mathcal{F}  (s, x) 
ds
\right|
&\leq & 
C\int_{\R^n}
\frac{1}{|x-y|^{n-2 \alpha}} 
\| \mathcal{F} (\cdot,y)\|_{L^\infty_t}
dy
\\
&
=
&
C
\mathcal{I}_{2\alpha}(
\| \mathcal{F} (\cdot,y)\|_{L^\infty_t}
)(x).
\end{eqnarray*} 
This last estimate implies 
\begin{equation}\label{eq.24.janvA}
\left\|\int_{0}^t\mathfrak{g}_{t-s}^\alpha\ast  f(s,x)ds\right\|_{L^\infty_t}\leq C\mathcal{I}_{2\alpha}(\|\mathcal{F}(\cdot,\cdot)\|_{L^\infty_t})(x).
\end{equation}
In order  to reconstruct the $\mathcal{L}^{p(\cdot)}_{\frac{nb}{2\alpha}}$-norm given in (\ref{Norm_PointFixe}), from the estimate \eqref{eq.24.janvA} we write 
\begin{eqnarray*}
\left\|\int_{0}^t\mathfrak{g}_{t-s}^\alpha\ast  f(s,x)ds\right\|_{L^{p(\cdot)}_x(L^\infty_t)}&\leq &C\left\|\mathcal{I}_{2\alpha}(\|\mathcal{F}(\cdot,\cdot)\|_{L^\infty_t})(\cdot)\right\|_{L^{p(\cdot)}_x},
\end{eqnarray*}
and 
\begin{eqnarray*}
\left\|\int_{0}^t\mathfrak{g}_{t-s}^\alpha\ast  f(s,x)ds\right\|_{L^{\frac{nb}{2\alpha}}_x(L^\infty_t)}&\leq &C\left\|\mathcal{I}_{2\alpha}(\|\mathcal{F}(\cdot,\cdot)\|_{L^\infty_t})(\cdot)\right\|_{L^{\frac{nb}{2\alpha}}_x}.
\end{eqnarray*}
Now, considering Proposition \ref{Proposition_RieszPotential}, we get the estimate
\begin{eqnarray}\label{eq 2 avril} 
\left\|
\int_{0}^t \mathfrak{g}^\alpha_{t-s}\ast  f  (s, x) 
ds
\right\|_{L^{p(\cdot)}_x(L^\infty_t)}
&\leq & 
C \left\|
\|\mathcal{F}\|_{L^\infty_t}\right\|_{\mathcal{L}^{\frac{p(\cdot)}{b+1}}_{   \frac{nb}{(b+1)(2\alpha )}     }   }
=
C
\|
\mathcal{F}
\|_{\mathcal{L}^{\frac{p(\cdot)}{b+1}}_{   \frac{nb}{(b+1)(2\alpha )}      ,x}(L^\infty_x)} 
,
\end{eqnarray}
and, since the Riesz potential $\mathcal{I}_{2\alpha }$  satisfies $\|\mathcal{I}_{2\alpha}(\varphi)\|_{L^{\frac{nb}{2\alpha }}}\leq C\|\varphi\|_{L^{\frac{nb}{(b+1)(2\alpha )}}  }$, we obtain 
\begin{eqnarray}\label{eq 3 avril} 
\left\|
\int_{0}^t \mathfrak{g}^\alpha_{t-s}\ast  f (s, x) 
ds
\right\|_{L^\frac{n b }{2\alpha }_x(L^\infty_t)}
&\leq &
C\left\| \| \mathcal{F} \|_{L^\infty_t}\right\|_{L^{ \frac{nb}{(b+1)(2\alpha )}   }_x} 
= 
C
\|\mathcal{F} \|_{L^\frac{nb}{(b+1)(2\alpha )}_x(L^\infty_t)}.
\end{eqnarray}
Then, 
gathering together 
the norm $\|\cdot\|_{\mathcal{E}}$ given in (\ref{Norm_PointFixe})
with the estimates 
\eqref{eq 2 avril} and
\eqref{eq 3 avril} , we get 
\begin{equation}\label{vase gamma 0 f}
    \left\|\int_{0}^t\mathfrak{g}_{t-s}^\alpha\ast  f(s,x)ds\right\|_{
    \mathcal{L}^{ {p(\cdot)} }_{\frac{nb}{ (2\alpha)},x}(L^\infty_t)
    } 
\leq C
\|\mathcal{F}\|_{
\mathcal{L}^{\frac{p(\cdot)}{b+1}}_{\frac{nb}{(b+1)(2\alpha)},x}(L^\infty_t)
}
.
\end{equation}

\end{itemize}

Considering the estimates \eqref{vase gamma 1 f} and \eqref{vase gamma 0 f}  we conclude 
\eqref{Control_f_preambulo} and thus the proof.

\end{proof}

\begin{Proposition}\label{prop.Control_Nolineal2}
Consider $\alpha\in ]\frac 1 2, 1]$ and a variable exponent $p(\cdot)\in \mathcal{P}^{\log}(\R^n)$ such that $p^->1$. Then, 
 there exist a constant $C_{\mathcal{B}}>0$ such that 
\begin{equation}\label{Control_Nolineal1}
\left\|
\int_{0}^t \mathfrak{g}^\alpha_{t-s}\ast  \vn^{\gamma} (|u|^b u ) (s, x) ds
\right\|_{\mathcal{E}}\leq C_{\mathcal{B}}\|u\|^b_{\mathcal{E}}\|u\|_{\mathcal{E}}.
\end{equation}
\end{Proposition}
\begin{proof}
In order to  conclude \eqref{Control_Nolineal1} in the next  we will study separately the cases when $\gamma=1$ and $\gamma=0$. 
\begin{itemize}

\item {\bf Case $\gamma=1$}.

We start the study of this case by noticing that, due the Minkowski’s integral inequality and  decay properties of the fractional heat kernel (see Remark \ref{remark 2.1 MIAO}), 
we can write 
\begin{eqnarray*}
\left|
\int_{0}^t \mathfrak{g}^\alpha_{t-s}\ast   \vec{1}\cdot \vn  (|u|^b u ) (s, x) ds
\right|
&\leq &
C\int_{0}^t\int_{\mathbb{R}^n}|\vn\mathfrak{g}_{t-s}^\alpha (x-y)|     |u(s,y)|^b  |u(s,y)|dyds
\\
&\leq &
C\int_{\mathbb{R}^n}\int_{0}^t
\frac{1}{   (\ |t-s|^{ \frac{1}{
       2\alpha }   } +
|x-y| \ )^{n+1}
}
|u(s,y)|^b |u(s,y)|dsdy.
\end{eqnarray*}
Then, considering the $L^\infty_t$-norm on the last inequality we obtain 
\begin{multline*}
\left|
\int_{0}^t
\mathfrak{g}^\alpha_{t-s}\ast   \vec{1}\cdot \vn  (|u|^b u ) (s, x) 
ds
\right|
\\\leq 
C\int_{\R^n}
\left(\int_{0}^t 
\frac{1}{  
 (\ |t-s|^{ \frac{1}{
       2\alpha }   } +
|x-y| \ )^{n+1}
}
ds\right)
\|u(\cdot,y)\|_{L^\infty_t}^b
\|u(\cdot,y)\|_{L^\infty_t}
dy.
\end{multline*}
Inspired by the proof of Proposition \ref{prop.Control_f_preambulo}  (see  \eqref{est 3.5}),  we can write 
\begin{equation}
\begin{aligned}
\int_0^t \frac{d s}{\left(|t-s|^{\frac{1}{2 \alpha}}+|x-y|\right)^{n+1}} 
& \leq  C \frac{1}{|x-y|^{n+1-2 \alpha}} 
\end{aligned}
\end{equation}
and then, we conclude the estimate
\begin{equation}
\begin{aligned}
\left|
\int_{0}^t
\mathfrak{g}^\alpha_{t-s}\ast   \vec{1}\cdot \vn  (|u|^b u ) (s, x) 
ds
\right|
&\leq & 
C
\int_{\mathbb{R}^n}
\frac{1}{|x-y|^{n+1-2 \alpha}} 
\|u(\cdot,y)\|_{L^\infty_t}^b
\|u(\cdot,y)\|_{L^\infty_t}
dy.
\end{aligned}
\end{equation}
Considering the definition of the Riesz Potential (see
\ref{Definition_RieszPotential}), the last inequality can be recast as 
$$\left|
\int_{0}^t \mathfrak{g}^\alpha_{t-s}\ast   \vec{1}\cdot \vn  (|u|^b u ) (s, x) 
ds
\right|
\leq C
\mathcal{I}_{2\alpha -1}
\big(\|u\|_{L^\infty_t}^b\|u\|_{L^\infty_t}\big)
(x).$$
To reconstruct the norm $\|\cdot\|_{\mathcal{E}}$ (see \eqref{Norm_PointFixe}) we write
\begin{eqnarray*}
\left\|
\int_{0}^t \mathfrak{g}^\alpha_{t-s}\ast   \vec{1}\cdot \vn  (|u|^b u ) (s, x) 
ds
\right\|_{L^{p(\cdot)}_x(L^\infty_t)}&
\leq&
C
\|\mathcal{I}_{2\alpha-1}
(\|u\|_{L^\infty_t}^b\|u\|_{L^\infty_t})\|_{L^{p(\cdot)}_x}
,
\end{eqnarray*}
and
\begin{eqnarray*}
{\color{black}
\left\|
\int_{0}^t \mathfrak{g}^\alpha_{t-s}\ast   \vec{1}\cdot \vn  (|u|^b u ) (s, x) 
ds
\right\|_{L^{\frac{nb}{2\alpha -1}}_x(L^\infty_t)}
}
&\leq &
C\|\mathcal{I}_{2\alpha-1}
(\|u\|_{L^\infty_t}^b\|u\|_{L^\infty_t})\|_{L^{\frac{nb}{2\alpha -  \langle1\rangle_\gamma  }}_x }.
\end{eqnarray*}
Then, considering that  $\|\mathcal{I}_{2\alpha-1}(\varphi)\|_{L^{\frac{nb}{2\alpha -  \langle1\rangle_\gamma  }}}\leq C\|\varphi\|_{L^{\frac{nb}{(b+1)(2\alpha -1)}}  }$, an H\"older inequality (see Remark \ref{Rem_Holder_Mixed_Lebesgue_Var}) and Proposition \ref{Proposition_RieszPotential}, we conclude 
\begin{equation}\label{eq AA}
\begin{aligned}
\left\|
\int_{0}^t \mathfrak{g}^\alpha_{t-s}\ast   \vec{1}\cdot \vn  (|u|^b u ) (s, x) 
ds
\right\|_{L^{p(\cdot)}_x(L^\infty_t)}
&\leq & 
C \left\|
\|u\|^b_{L^\infty_t}\|u\|_{L^\infty_t}\right\|_{\mathcal{L}^{\frac{p(\cdot)}{b+1}}_{   \frac{nb}{(b+1)(2\alpha -1)}     }   }\\
&\leq & 
C\|u\|^b_{\mathcal{L}^{p(\cdot)}_{   \frac{nb}{2\alpha -  \langle1\rangle_\gamma  }   ,x}(L^\infty_x)}
\|
u
\|_{\mathcal{L}^{p(\cdot)}_{   \frac{nb}{2\alpha -  \langle1\rangle_\gamma  }      ,x}(L^\infty_x)} 
,
\end{aligned}
\end{equation}
and
\begin{equation}\label{eq BB}
\begin{aligned}
\left\|
\int_{0}^t \mathfrak{g}^\alpha_{t-s}\ast   \vec{1}\cdot \vn  (|u|^b u ) (s, x) 
ds
\right\|_{L^\frac{n b }{2\alpha -1}_x(L^\infty_t)}
&\leq &
C\left\|\|u\|^b_{L^\infty_t}\|u\|_{L^\infty_t}\right\|_{L^{ \frac{nb}{(b+1)(2\alpha -1)}   }_x}
\\
&\leq & 
C\|u\|^b_{L^\frac{nb}{2\alpha -  \langle1\rangle_\gamma  }_x(L^\infty_t)}\|u\|_{L^\frac{nb}{2\alpha -  \langle1\rangle_\gamma  }_x(L^\infty_t)}.
\end{aligned}
\end{equation}
\begin{Remarque}\label{Rem_Riesz_MixedLebesgue}
Note that, in the case that we had considered Theorem \ref{theo.PotentialRieszVariable0} instead of Proposition \ref{Proposition_RieszPotential}, we had obtained an estimate of the form 
$\|\mathcal{I}_{2\alpha-1}(\varphi)\|_{L^{p(\cdot)}}\leq \|\varphi\|_{L^{\frac{p(\cdot)}{b+1}}}$, 
which, due the strong relationship between the the variable exponents involved, yields the constant exponent
$p(\cdot)\equiv \frac{nb}{2\alpha-1}$.
\end{Remarque}
Gathering together the estimates 
\eqref{eq AA}, 
\eqref{eq BB}, and using the definition of the norm $\|\cdot\|_{\mathcal{E}}$ given in (\ref{Norm_PointFixe}) we get the estimate  
\begin{equation}\label{vase gamma 1 ff}
   \left\|\int_{0}^t\mathfrak{g}_{t-s}^\alpha\ast  \vec{1}\cdot \vn  (|u|^b u ) 
   (s,x)ds\right\|_{
\mathcal{L}^{ {p(\cdot)} }_{\frac{nb}{ (2\alpha-1)},x}(L^\infty_t)
   } 
\leq C
\|u\|^b_{
\mathcal{L}^{ {p(\cdot)} }_{\frac{nb}{ (2\alpha-1)},x}(L^\infty_t)
}
\|u\|_{
\mathcal{L}^{ {p(\cdot)} }_{\frac{nb}{ (2\alpha-1)},x}(L^\infty_t)
}.
\end{equation}

\item {\bf Case $\gamma=0$}. By considering the Minkowski’s integral inequality and the  
  decay properties of the fractional heat kernel (see Remark \ref{remark 2.1 MIAO}), 
    we get the estimates 
\begin{eqnarray*}
\left|
\int_{0}^t \mathfrak{g}^\alpha_{t-s}\ast  (|u|^b u ) (s, x) ds
\right|
&\leq &
C\int_{\mathbb{R}^n}\int_{0}^t
\frac{|t-s|}{   (\ |t-s|^{ \frac{1}{
       2\alpha }   } +
|x-y| \ )^{n+2\alpha}
}
|u(s,y)|^b |u(s,y)|dsdy.
\end{eqnarray*}
Considering $L^\infty_t$-norm  we conclude 
\begin{multline*}
\left|
\int_{0}^t
\mathfrak{g}^\alpha_{t-s}\ast   (|u|^b u ) (s, x) 
ds
\right|
\\\leq 
C\int_{\R^n}
\left(\int_{0}^t 
\frac{|t-s|}{  
 (\ |t-s|^{ \frac{1}{
       2\alpha }   } +
|x-y| \ )^{n+2\alpha}
}
ds\right)
\|u(\cdot,y)\|_{L^\infty_t}^b
\|u(\cdot,y)\|_{L^\infty_t}
dy.
\end{multline*}
Now, considering the definition of the Riesz potential, and the estimate (see \eqref{est 3.10})
\begin{equation}
\begin{aligned}
\int_0^t \frac{ |t-s| ds}{\left(|t-s|^{\frac{1}{2 \alpha}}+|x-y|\right)^{n+2\alpha}} 
&\leq \frac{1}{|x-y|^{n-2 \alpha}} 
\int_0^{+\infty} \frac{\beta d \beta}{\left(1+\beta^{\frac{1}{2 \alpha}}\right)^{n+2\alpha}},
\end{aligned}
\end{equation}
we obtain the estimate 
\begin{eqnarray*}
\left|
\int_{0}^t
\mathfrak{g}^\alpha_{t-s}\ast   (|u|^b u ) (s, x) 
ds
\right|
&\leq & 
C\int_{\R^n}
\frac{1}{|x-y|^{n-2 \alpha}} 
\|u(\cdot,y)\|_{L^\infty_t}^b
\|u(\cdot,y)\|_{L^\infty_t}
dy
\\
&
=
&
C
\mathcal{I}_{2\alpha}(
\|u(\cdot,y)\|_{L^\infty_t}^b
\|u(\cdot,y)\|_{L^\infty_t}
)(x).
\end{eqnarray*} 
In order o reconstruct the norm $\|\cdot\|_{\mathcal{E}}$ given in \eqref{Norm_PointFixe} we can  write
\begin{eqnarray*}
\left\|
\int_{0}^t \mathfrak{g}^\alpha_{t-s}\ast  (|u|^b u ) (s, x) 
ds
\right\|_{L^{p(\cdot)}_x(L^\infty_t)}&
\leq&
C
\|\mathcal{I}_{2\alpha}
(\|u\|_{L^\infty_t}^b\|u\|_{L^\infty_t})\|_{L^{p(\cdot)}_x}
,
\end{eqnarray*}
and
\begin{eqnarray*}
\left\|
\int_{0}^t \mathfrak{g}^\alpha_{t-s}\ast  (|u|^b u ) (s, x) 
ds
\right\|_{L^\frac{n b }{2\alpha }_x(L^\infty_t)}
&\leq &
C\|\mathcal{I}_{2\alpha}
(\|u\|_{L^\infty_t}^b\|u\|_{L^\infty_t})\|_{L^{\frac{nb}{2\alpha}}_x }.
\end{eqnarray*}

If we consider an H\"older inequality (see Remark \ref{Rem_Holder_Mixed_Lebesgue_Var}),  Proposition \ref{Proposition_RieszPotential} and the fact that the Riesz potential is bounded in ${L^{\frac{nb}{(b+1)(2\alpha )}}  }$, we get the estimates 
\begin{equation}\label{eq A}
\begin{aligned}
\left\|
\int_{0}^t \mathfrak{g}^\alpha_{t-s}\ast  (|u|^b u ) (s, x) 
ds
\right\|_{L^{p(\cdot)}_x(L^\infty_t)}
&\leq & 
C \left\|
\|u\|^b_{L^\infty_t}\|u\|_{L^\infty_t}\right\|_{\mathcal{L}^{\frac{p(\cdot)}{b+1}}_{   \frac{nb}{(b+1)(2\alpha )}     }   }
\\
&\leq&
C\|u\|^b_{\mathcal{L}^{p(\cdot)}_{   \frac{nb}{2\alpha }   ,x}(L^\infty_x)}
\|
u
\|_{\mathcal{L}^{p(\cdot)}_{   \frac{nb}{2\alpha }      ,x}(L^\infty_x)} 
,
\end{aligned}
\end{equation}

and
\begin{equation}\label{eq B}
\begin{aligned}
\left\|
\int_{0}^t \mathfrak{g}^\alpha_{t-s}\ast  (|u|^b u ) (s, x) 
ds
\right\|_{L^\frac{n b }{2\alpha }_x(L^\infty_t)}
&\leq &
C\left\|\|u\|^b_{L^\infty_t}\|u\|_{L^\infty_t}\right\|_{L^{ \frac{nb}{(b+1)(2\alpha )}   }_x}
\\ &\leq &
C\|u\|^b_{L^\frac{nb}{2\alpha }_x(L^\infty_t)}\|u\|_{L^\frac{nb}{2\alpha }_x(L^\infty_t)}.
\end{aligned}
\end{equation}
Thus, gathering 
\eqref{eq A}, 
\eqref{eq B}, and using the definition of the norm $\|\cdot\|_{\mathcal{E}}$ given in (\ref{Norm_PointFixe}), we obtain 
\begin{equation}\label{vase gamma 0 ff}
   \left\|\int_{0}^t\mathfrak{g}_{t-s}^\alpha\ast  (|u|^b u ) 
   (s,x)ds\right\|_{
\mathcal{L}^{ {p(\cdot)} }_{\frac{nb}{ (2\alpha)},x}(L^\infty_t)
   } 
\leq C
\|u\|^b_{
\mathcal{L}^{ {p(\cdot)} }_{\frac{nb}{ (2\alpha)},x}(L^\infty_t)
}
\|u\|_{
\mathcal{L}^{ {p(\cdot)} }_{\frac{nb}{ (2\alpha)},x}(L^\infty_t)
}.
\end{equation}

\end{itemize}

Considering the estimates \eqref{vase gamma 1 ff} and \eqref{vase gamma 0 ff}  we conclude 
\eqref{Control_Nolineal1}, and then  Proposition \ref{prop.Control_Nolineal} is proven.

\end{proof}

\subsection*{End of the proof of Theorem \ref{Theoreme_2}}
Note that, the integral equation \eqref{NS_Integral} can be rewritten as $ u= \Psi (u)$, where $\Psi$ is given by
\begin{equation*}
    \Psi (u)
    :=
    \mathfrak{g}^\alpha_t\ast u_0(x)
+
\int_{0}^t
\mathfrak{g}^\alpha_{t-s}\ast\vf(s, x)
ds
-
\int_{0}^t
\mathfrak{g}^\alpha_{t-s}\ast 
\vn^{\gamma} (|u|^b u )
(s, x)
ds.
\end{equation*}
 The operator $\Psi$ is seen as a map from the space $\mathcal{E}$ into itself. Let us denote by
\begin{equation*}
    B = C\|u_0\|_{\mathcal{L}^{p(\cdot)}_{   \frac{nb}{2\alpha -  \langle1\rangle_\gamma  }        } }
    +
C
\|\mathcal{F}\|_{
\mathcal{L}^{\frac{p(\cdot)}{b+1}}_{\frac{nb}{(b+1)(2\alpha-
\langle 1 \rangle_\gamma
)},x}(L^\infty_t)
} ,
\end{equation*}
and set 
{\color{black}$R= 2B$}. With these notations at hand, now we consider the ball 
\begin{equation*}
    B_{R}
    =
    \left\{
     u \in \mathcal{E}\, :\,
     \|
    u
     \|_{
\mathcal{E}
    }
    \leq R
    \right\}.
\end{equation*}
In the next, we  will prove that if $B$ is bounded by an appropriate constant, then the operator $\Psi$ is a contraction map on $B_R$.
To this end, in the next we consider $u,\ v \in B_R$ and a fixed parameter $\gamma \in \{0,1\}$. 
Thus, by the estimate obtained in Proposition   
 \ref{prop.Control_Nolineal2},
we obtain the following inequalities
\begin{equation}\label{26fevr3}
    \begin{split}
    \|\Psi (u)
    -
    \Psi ( v) \|_{
\mathcal{E}}
&    \leq 
    C
    \|
    u
    -
     v \|_{
\mathcal{E}
    }
    \left(
 \|
    u
     \|^b_{
\mathcal{E}
    }
    +
     \|
     v
     \|^b_{
\mathcal{E}
    }
    \right)
 \\ &   \leq 
    CR^b 
     \|
    u
    -
     v \|_{
\mathcal{E}
    }
    .
 \end{split}
\end{equation}
To continue, by Propositions  \ref{prop.Control_Uo} and \ref{prop.Control_f_preambulo},
we can write 
\begin{equation}\label{26fevr4}
    \begin{split}
    \|\Psi (u)
 \|_{
\mathcal{E}}
&
    \leq 
    \|\Psi (u) - \Psi (0) + \Psi (0)
 \|_{
\mathcal{E}}
\\
&
    \leq 
    \|\Psi (u) - \Psi (0) 
 \|_{
\mathcal{E}}
+
    \| \Psi (0)
 \|_{
\mathcal{E}}
\\
&
\leq 
CR^b  \| u
 \|_{
\mathcal{E}} +
B. 
 \end{split}
\end{equation}
Now, if $2CB \leq \frac{1}{2}$, then $CR^b \leq  \frac{1}{2} $, and thus, considering \eqref{26fevr3} and \eqref{26fevr4}, we conclude  
\begin{equation*}
        \|\Psi (u)
    -
    \Psi ( v) \|_{
\mathcal{E}} 
\leq
\frac{1}{2} \|
    u
    -
     v \|_{
\mathcal{E}
    }
\quad \text{and} \quad
        \|\Psi (u)
    \|_{
\mathcal{E}} 
\leq
R. 
\end{equation*}
Thus, it follows from the contraction mapping principle (stated in Proposition \ref{cmp}) that there exists a unique mild solution $u \in \mathcal{E}$ of the equation \eqref{NS_Intro}. 
With this we conclude the proof of Theorem \ref{Theoreme_2}.

\subsection{Proof of Theorem \ref{Theoreme_1}}
In this subsection we consider a variable exponent $p(\cdot)\in \mathcal{P}^{log}([0,+\infty[)$ such that $1<p^-\leq p^+<+\infty$, and the functional space 
$$
 E_T=
 L^{p(\cdot)} \left( [0,T], L^{q} (\Rt) \right),$$  
with $T\in ]0,+\infty[ $ to be precised later. 
The space $ E_T$ is dotted with the following Luxemburg-type norm: 
\begin{equation}\label{Norm_PointFixe_2}
\| \vec{\varphi}\|_{ E_T}=\inf\left\{\lambda > 0: \,\int_0^{T}\left|\frac{ \|
\vec{\varphi} (t,\cdot)\|_{L^{q}}}{\lambda}\right|^{p(t)} dt \leq1\right\}.
\end{equation}
In this subsection we attack the proof following the same strategy that in Theorem \ref{Theoreme_2}, we will  construct mild solutions for the integral equation  (\ref{NS_Integral}) by considering the 
mapping contraction principle on the space $E_T$. 
\\

\noindent To this end, we begin by proving the following result regarding a control of the initial data. 

\begin{Proposition}\label{prop.Control_Uo_Lplq}
Let $\alpha\in ]\frac 1 2, 1],\ p(\cdot )\in \mPl(\Rt)$ with $b+1<p^-\leq p^+ <+\infty$, fix an index $q>\frac{nb}{2\alpha -  \langle1\rangle_\gamma  }$ by the relationship
$\frac{\alpha b}{p(\cdot)}
+
\frac{nb}{2q}
<\alpha - \langle \frac{1}{2} \rangle_\gamma$
and $\overline{q}(\cdot)\in\mP^{\text{emb}}_q(\Rt)$.
Consider a function
$u_0\in L^{ \overline{q}(\cdot) } (\Rt)$. Then, there exist a constant $C>0$ such that 
\begin{equation}\label{Control_Uo_Lplq}
\|\mathfrak{g}_t^\alpha * u_0  \|_{  E_T}\leq C \| u_0  \|_{L^{
\overline{q}(\cdot)
}}.
\end{equation} 
\end{Proposition}
\begin{proof}
To deduce
the inequality \eqref{Control_Uo_Lplq},   we begin by considering the Young inequality to obtain 
$$\|\mathfrak{g}_t^\alpha * u_0  \|_{L^{q} (\Rt)} \leq \|\mathfrak{g}_t^\alpha   \|_{L^{1} (\Rt)}\| u_0  \|_{L^{q} (\Rt)} = \| u_0  \|_{L^{q} (\Rt)}.$$ 
Now, let recall the following result in the context of Variable Lebesgue spaces (see \cite[Lemma 3.2.12, Section 3.2]{Diening_Libro}).
\begin{Lemme}\label{Lemma_subset}
Let $p(\cdot)\in \mathcal{P}([0,+\infty[)$ such that $1<p^-\leq p^+<+\infty$. Then, there exist $C>0$ such that 
$$
\|1\|_{L^{p(\cdot)}([0,T])} \leq
C \max\left\{T^{\frac{1}{p^{-}}}, T^{\frac{1}{p^{+}}}\right\}.$$
\end{Lemme}
To continue, we consider  
 the $L^{p(\cdot)}_t$-norm on the last inequality to get 
\begin{eqnarray*}
\|\mathfrak{g}_t^\alpha * u_0  \|_{L^{p(\cdot)} _tL^{q}_x }
&\leq& C\|u_0\|_{L^{q} (\Rt)}\|1\|_{L^{p(\cdot)}([0,T])} 
\\
&\leq& C \max\left\{T^{\frac{1}{p^{-}}}, T^{\frac{1}{p^{+}}}\right\} \|u_0\|_{L^{q} (\Rt)}.
\end{eqnarray*}
Then, by Lemma \ref{lema embeddd set} we conclude the estimate
\begin{equation*}
\|\mathfrak{g}^\alpha_t *  u_0  \|_{L^{p(\cdot)} _tL^{q}_x }\leq C \max\left\{T^{\frac{1}{p^{-}}}, T^{\frac{1}{p^{+}}}\right\}
\| u_0\|_{L^{\overline{q}(\cdot)} (\Rt)}
.
\end{equation*}
With this we conclude the proof of Proposition \ref{prop.Control_Uo_Lplq}. 
\end{proof}
\begin{Proposition}\label{prop.Control_force_LpLq}
Let $\gamma \in \{0,1\}$ fixed.  Consider $\alpha\in (\frac 1 2, 1)$, a variable exponent $ p(\cdot )\in \mPl(\Rt)$ with $b+1<p^-\leq p^+ <+\infty$, fix an index $q>\frac{nb}{2\alpha -  \langle1\rangle_\gamma  }$ by the relationship
$\frac{\alpha b}{p(\cdot)}
+
\frac{nb}{2q}
<\alpha - \langle \frac{1}{2} \rangle_\gamma$
and $\overline{q}(\cdot)\in\mP^{\text{emb}}_q(\Rt)$.
Then, given a function
 $\vf\in L^{p(\cdot)} \left( [0,+\infty[,  L^{\overline{q}(\cdot)} (\R^3) \right)$, there exist a numerical constant $C>0$ such that
\begin{equation} \label{Control_force_LpLq}
\left\|\int_0^t\mathfrak{g}_{t-s}^\alpha *\vf(s, \cdot) ds\right\|_{ E_T} \leq 
C\|\vf \|_{L^1_tL^{\overline{q}(\cdot)}_x}.
\end{equation} 
\end{Proposition}
\begin{proof}
We begin by taking the $L^q_x$-norm 
to the term $\int_0^t\mathfrak{g}_{t-s}^\alpha *\vf(s, x) ds$
in order to obtain the estimate 
$$\left\|\int_0^t\mathfrak{g}^\alpha_{t-s} \ast \vf (s, \cdot)ds\right\|_{L^{q}} \leq  \int_0^t\|\mathfrak{g}^\alpha_{t-s}\|_{L^1} \| \vf(s, \cdot) \|_{L^{q}}ds .$$
Then, by considering Lemma \ref{lema embeddd set} we can write 
$$\left\|\int_0^t\mathfrak{g}^\alpha_{t-s} \ast \vf (s, \cdot)ds\right\|_{L^{q}} \leq  \int_0^t\|\mathfrak{g}^\alpha_{t-s}\|_{L^1} \| \vf(s, \cdot) \|_{L^{  \overline{q}(\cdot)  }}ds \leq C\| \vf\|_{L^1_t L^{
\overline{q}(\cdot)
}_x }.$$
Proceeding as in the previous lines, \emph{i.e.} taking the $L^{p(\cdot)}$-norm in the time variable and using Lemma \ref{Lemma_subset}, we get the estimates 
\begin{eqnarray*}
\left\|\int_0^t\mathfrak{g}^\alpha_{t-s} \ast \vf(s, \cdot) ds\right\|_{L^{p(\cdot)}_t(L^{q}_x)}
&\leq &
C\left\|\| \vf  \|_{L^1_t L^{\overline{q}(\cdot)}_x }\right\|_{L^{p(\cdot)}_t}
\\
&\leq &
 C\| \vf\|_{L^1_t L^{ \overline{q}(\cdot) }_x}\|1\|_{L^{p(\cdot)}([0,T])}\notag
\\
&\leq &
C \max\left\{T^{\frac{1}{p^{-}}}, T^{\frac{1}{p^{+}}}\right\}
\| \vf\|_{L^1_t L^{  \overline{q}(\cdot)  }_x}.\label{Estimation_ForceExterieure}
\end{eqnarray*}
Considering the last estimate, we conclude \eqref{Control_force_LpLq}  
and we finish the proof of Proposition \ref{prop.Control_force_LpLq}. 
\end{proof}
\begin{Proposition}\label{prop.Control_Nolineal}Let $\gamma \in \{0,1\}$ fixed.  
Consider  $\alpha\in ]\frac 1 2, 1]$, a variable exponent $ p(\cdot )\in \mPl(\Rt)$ with $b+1<p^-\leq p^+ <+\infty$,  fix an index $q>\frac{nb}{2\alpha -  \langle1\rangle_\gamma  }$ by the relationship
$\frac{\alpha b}{p(\cdot)}
+
\frac{nb}{2q}
<\alpha - \langle \frac{1}{2} \rangle_\gamma$, and $\overline{q}(\cdot)\in\mP^{\text{emb}}_q(\Rt)$.
Then, 
there exist a constant $C_{\mathcal{B}}>0$ such that
\begin{equation}\label{Control_Nolineal}
\left\|
\int_{0}^t \mathfrak{g}^\alpha_{t-s}\ast  \vn^{\gamma} (|u|^b u ) (s, x) ds
\right\|_{ E_T}\leq C_{\mathcal{B}}\|u\|^b_{ E_T}\|u\|_{ E_T}.
\end{equation}
\end{Proposition}
\begin{proof}
We begin the proof  by considering
the $L^q_x$-norm to the term $\int_{0}^t \mathfrak{g}^\alpha_{t-s}\ast  \vn^{\gamma} (|u|^b u ) (s, x) ds$ in order  
to obtain 
$$\left\|
\int_{0}^t \mathfrak{g}^\alpha_{t-s}\ast  \vn^{\gamma} (|u|^b u ) (s, x) ds
\right\|_{L^q}\leq C\int_{0}^t\left\|
\mathfrak{g}^\alpha_{t-s}\ast  \vn^{\gamma} (|u|^b u ) (s, \cdot)
\right\|_{L^q}ds.$$
From this inequality we can conclude 
$$\displaystyle \left\|\int_{0}^t \mathfrak{g}^\alpha_{t-s}\ast  \vn^{\gamma} (|u|^b u ) (s, x) ds
\right\|_{L^q}\leq 
\begin{cases}
   \displaystyle C\int_0^t \left\|\mathfrak{g}_{t-s}^\alpha * (|u|^b u ) (s, \cdot)\right\|_{L^q} ds & \text{if } \gamma=0,
    \\[11pt]
   \displaystyle C\int_0^t \left\|\vn \mathfrak{g}_{t-s}^\alpha *(|u|^b u ) (s, \cdot)\right\|_{L^q} ds
   & \text{if } \gamma=1.
\end{cases}
$$
Note that, Lemma \ref{Lemma 3.1 MIAO}  yields the estimate
\begin{eqnarray*}
\int_0^t \left\|\vn \mathfrak{g}_{t-s}^\alpha * u\otimes u(s, \cdot)\right\|_{L^q} ds
&\leq&
  \displaystyle  C\int_0^t  
 \frac{1}{
 (t-s)^{
\langle
 \frac{1}{2\alpha}
\rangle_\gamma 
 +\frac{nb}{2\alpha q}
 }
 }
\| |u|^b u (s,\cdot)\|_{L^{ \frac{q}{b+1} }}   ds.
\end{eqnarray*}
Thus, considering an H\"older   inequality  we obtain 
 \begin{eqnarray*}
\int_0^t \left\|\vn \mathfrak{g}_{t-s}^\alpha * u\otimes u(s, \cdot)\right\|_{L^q} ds
&\leq&
  \displaystyle  C\int_0^t  
 \frac{1}{
 (t-s)^{
    \langle \frac{1}{2\alpha}\rangle_\gamma
 +\frac{nb}{2\alpha q}
 }
 }
\|u(s,\cdot)\|_{L^{q}}^b
\|u(s,\cdot)\|_{L^{q}} ds.
\end{eqnarray*}
Now, considering the $L^{p(\cdot)}_t$-norm
and the norm conjugate formula (\ref{Norm_conjugate_formula}), provided with the 
 variable exponent $p'(\cdot)$ defined by $1=\frac{1}{p(\cdot)}+\frac{1}{p'(\cdot)}$, 
we get 
\begin{eqnarray}
\left\|\int_{0}^t \mathfrak{g}^\alpha_{t-s}\ast  \vn^{\gamma} (|u|^b u ) (s, x) ds
\right\|_{L_t^{p(\cdot)}(L_x^q)}
&\leq&
\left\|\int_0^t 
\frac{1}{
 (t-s)^{
    \langle \frac{1}{2\alpha}\rangle_\gamma
 +\frac{nb}{2\alpha q}
 }
 }
\left\|
u(s,\cdot)
\right\|_{L^{q}}^{b+1} ds\right\|_{L_t^{p(\cdot)}([0,T])}\label{Identite_NormeDualite}
\\
&\leq &
\sup_{
\| \psi \|_{L^{p'(\cdot)}}\leq 1
} 
\int_0^{T} \int_0^t 
\frac{|\psi(t)|}{
|t-s|^{
    \langle \frac{1}{2\alpha}\rangle_\gamma
 +\frac{nb}{2\alpha q}
}}
\left\|u(s,\cdot)\right\|_{L^{q}}^{b+1}
ds \,dt.\notag
\end{eqnarray}
Note that, considering Fubini's Theorem, we can write 
\begin{multline*}
    \sup_{\| \psi \|_{L^{p'(\cdot)}}\leq 1}
\int_0^{T} \int_0^t 
\frac{|\psi(t)|}{
|t-s|^{
    \langle \frac{1}{2\alpha}\rangle_\gamma
 +\frac{nb}{2\alpha q}
}}\|u(s,\cdot)
\|^{b}_{L^q}
\|u(s,\cdot)
\|_{L^q}
ds \,dt
\\=
\sup_{\| \psi \|_{L^{p'(\cdot)}}\leq 1}
\int_0^T\int_0^T
\frac{1_{ \{ 0<s<t \} } |\psi(t)| }{|t-s|^{
    \langle \frac{1}{2\alpha}\rangle_\gamma
 +\frac{nb}{2\alpha q}
}}
dt
\|u(s,\cdot)
\|^{b}_{L^q}
\|u(s,\cdot)
\|_{L^q}
ds,
\end{multline*}
Now, extending the function $\psi(t)$ by zero 
on $\mathbb{R} \setminus [0,T]$, we obtain 
\begin{multline*}
\sup_{\| \psi \|_{L^{p'(\cdot)}}\leq 1}
\int_0^{T} \int_0^t 
\frac{|\psi(t)|}{|t-s|^{
    \langle \frac{1}{2\alpha}\rangle_\gamma
 +\frac{nb}{2\alpha q}
}}
\|u(s,\cdot)
\|^{b}_{L^q}
\|u(s,\cdot)
\|_{L^q}
ds \,dt
\\ =
\sup_{\| \psi \|_{L^{p'(\cdot)}}\leq 1}
\int_0^T
\left(\int_{-\infty}^{+\infty}\frac{|\psi(t)| }{
|t-s|^{
    \langle \frac{1}{2\alpha}\rangle_\gamma
 +\frac{nb}{2\alpha q}
}
}
dt
\right)
\|u(s,\cdot)
\|^{b}_{L^q}
\|u(s,\cdot)
\|_{L^q}
ds
\\
=\sup_{\| \psi \|_{L^{p'(\cdot)}}\leq 1}\int_0^T\mathcal{I}_\beta(|\psi|)(s)\|u(s,\cdot)
\|^{b}_{L^q}
\|u(s,\cdot)
\|_{L^q}ds,
\end{multline*}
where $\mathcal{I}_\beta$ is the 1D Riesz potential   with $\beta=1-  \langle \frac{1}{2\alpha}\rangle_\gamma
 -\frac{nb}{2\alpha q}
<1$ (see Definition \ref{Definition_RieszPotential}). 
\begin{Remarque}
Note that the constraints  $\frac{nb}{2\alpha -  \langle1\rangle_\gamma  }<q$ and $\alpha \in ] \frac 1 2, 1]$ imply $0<1-  \langle \frac{1}{2\alpha}\rangle_\gamma
 -\frac{nb}{2\alpha q}<1$, and then the Riesz potential considered is well defined.
\end{Remarque}
Considering an H\"older inequality with $1=\frac{b}{p(\cdot)}+\frac{1}{p(\cdot)}+\frac{1}{ \tilde{p}(\cdot)}$, we get the estimate 
\begin{multline*}
    \sup_{\| \psi \|_{L^{p'(\cdot)}}\leq 1}\int_0^T\mathcal{I}_\beta(|\psi|)(s)
\|u(s,\cdot)
\|^{b}_{L^q}
\|u(s,\cdot)
\|_{L^q}ds
\\
\leq
C\sup_{\| \psi \|_{L^{p'(\cdot)}}
\leq 1}\|\mathcal{I}_\beta(|\psi|)\|_{L_t^{\tilde{p}(\cdot)}}
\Big\|\|u(\cdot,\cdot)\|_{L_x^q}\Big\|^b_{L_t^{p(\cdot)}}
\Big\|\|u(\cdot,\cdot)\|_{L_x^q}\Big\|_{L_t^{p(\cdot)}}
.
\end{multline*}
\begin{Remarque} 
We must stress the fact that condition $p^->b+1$ in the statement of the proposition become from the  relationship $1=\frac{b+1}{p(\cdot)}+\frac{1}{ \tilde{p}(\cdot)}$ .
\end{Remarque}
Now, provided of  Theorem \ref{theo.PotentialRieszVariable0} with indexes defined by the relationship
\begin{equation}\label{Indices_Riesz}
\frac{1}{\tilde{p}(\cdot)}
=
\frac{1}{r(\cdot)}-
\left(
1-  \langle \frac{1}{2\alpha}\rangle_\gamma
 -\frac{nb}{2\alpha q}
\right),
\end{equation}
we conclude the estimate 
\begin{multline}
\sup_{\| \psi \|_{L^{p'(\cdot)}}\leq 1}\|\mathcal{I}_\beta(|\psi|)\|_{L_t^{\tilde{p}(\cdot)}}\|u(\cdot,\cdot)\|_{L_t^{p(\cdot)}(L_x^q)}^{b}
\|u(\cdot,\cdot)\|_{L_t^{p(\cdot)}(L_x^q)}
\\
\leq C\sup_{\| \psi \|_{L^{p'(\cdot)}}\leq 1}\|\psi\|_{L_t^{r(\cdot)}}
\|u(\cdot,\cdot)\|^b_{L_t^{p(\cdot)}(L_x^q)}\|u(\cdot,\cdot)\|_{L_t^{p(\cdot)}(L_x^q)}.\label{Estimation_AvantInclusion}
\end{multline}
By hypothesis we know that $\frac{\alpha b}{p(\cdot)}
+
\frac{nb}{2q}
<\alpha - \langle \frac{1}{2} \rangle_\gamma$, then from relationship  \eqref{Indices_Riesz} and the expressions  
$$
\frac{1}{\tilde{p}(\cdot)}=1-\frac{b+1}{p(\cdot)}
\quad \text{and} \quad 
\frac{1}{p'(\cdot)}=1-\frac{1}{p(\cdot)},
$$ 
 we deduce $r(\cdot)<p'(\cdot)$. 
 Thus, by
Considering Lemma \ref{lemme_embeding} (with $r(\cdot)<p'(\cdot)$ and $\Omega= [0,T]$ ), from (\ref{Estimation_AvantInclusion}) we obtain 
\begin{eqnarray*}
\sup_{\| \psi \|_{L^{p'(\cdot)}}\leq 1}\|\psi\|_{L_t^{r(\cdot)}}\|u(\cdot,\cdot)\|^b_{L_t^{p(\cdot)}(L_x^q)}\|u(\cdot,\cdot)\|_{L_t^{p(\cdot)}(L_x^q)}\qquad\qquad\qquad\qquad\qquad\qquad\\
\leq \sup_{\| \psi \|_{L^{p'(\cdot)}}\leq 1}(1+T)\|\psi\|_{L_t^{p'(\cdot)}}\|u(\cdot,\cdot)\|^b_{L_t^{p(\cdot)}(L_x^q)}\|u(\cdot,\cdot)\|_{L_t^{p(\cdot)}(L_x^q)}\\
\leq (1+T)\|u(\cdot,\cdot)\|^b_{L_t^{p(\cdot)}(L_x^q)}\|u(\cdot,\cdot)\|_{L_t^{p(\cdot)}(L_x^q)}.
\end{eqnarray*}
Gathering these last estimates with (\ref{Identite_NormeDualite}), we conclude 
\begin{equation}\label{Estimation_ApplicationBilineaire}
\left\|\int_{0}^t \mathfrak{g}^\alpha_{t-s}\ast  \vn^{\gamma} (|u|^b u ) (s, x) ds
\right\|_{L_t^{p(\cdot)}(L_x^q)}
\leq C(1+T)\|u(\cdot,\cdot)\|^b_{L_t^{p(\cdot)}(L_x^q)}\|u(\cdot,\cdot)\|_{L_t^{p(\cdot)}(L_x^q)}.
\end{equation}
Considering this last estimate, we deduce the inequality \eqref{Control_Nolineal} and we finish the proof. 

\end{proof}

\subsection*{End of the proof of Theorem \ref{Theoreme_1}}
Motivated by the integral equation \eqref{NS_Integral}, we consider the operator $\Phi$ defined  by the expression
\begin{equation*}
    \Phi (u)
    :=
    \mathfrak{g}^\alpha_t\ast u_0(x)
+
\int_{0}^t
\mathfrak{g}^\alpha_{t-s}\ast\vf(s, x)
ds
-
\int_{0}^t
\mathfrak{g}^\alpha_{t-s}\ast 
\vn^{\gamma} (|u|^b u )
(s, x)
ds,
\end{equation*}
which maps $ E_T$ into itself. Now, denoting by 
\begin{equation*}
 {\color{black}   B = C\max\left\{T^{\frac{1}{p^{-}}}, T^{\frac{1}{p^{+}}}\right\} 
\left(
\|u_0\|_{L^{\overline{q}} (\Rt)}
    +
\| \vf\|_{L^1_t(L^{\overline{q}}_x)} 
\right),
}
\end{equation*}
and setting $ {\color{black}R= 2B}$,  we introduce  the ball 
\begin{equation*}
    B_{R}
    =
    \left\{
     u \in  E_T \, :\,
     \|
    u
     \|_{
 E_T
    }
    \leq R
    \right\}.
\end{equation*}
Then, similarly to the strategy considered in the proof of Theorem \ref{Theoreme_2}, in the following we will prove that if $B$ is bounded by a suitable constant, then the operator $\Phi$ is a contraction map on the ball $B_R$. 
To continue, we consider a fixed parameter $\gamma \in \{0,1\}$ and $u,v\in B_R$. Then, by the result obtained in Proposition
 \ref{prop.Control_Nolineal}, we obtain the following estimate
\begin{equation}\label{26fevr1}
    \|\Phi (u)
    -
    \Phi ( v) \|_{
 E_T
    }
    \leq 
{\color{black}    C
    \| u- v \|_{E_T}
    \left(
 \|
    u
     \|^b_{
 E_T
    }
    +
     \|
    v
     \|^b_{
 E_T
    } 
    \right)}
    \leq 
    C R^b
    \| u- v \|_{E_T}  
.
\end{equation}
With this information at hand, and considering  Propositions 
 \ref{prop.Control_Uo_Lplq} and \ref{prop.Control_force_LpLq}, we can write  
\begin{equation}\label{26fevr2}
\begin{split}
    \|\Phi (u)
 \|_{
{ E_T}}
&
    \leq 
    \|\Phi (u) - \Phi (0) + \Phi (0)
 \|_{
{ E_T}}
\\ &
    \leq 
    \|\Phi (u) - \Phi (0) \|_{E_T}
+
    \| \Phi (0)
 \|_{
{ E_T}}
\\ &
\leq   
CR^b \| u \|_{E_T}
+
C\max\left\{T^{\frac{1}{p^{-}}}, T^{\frac{1}{p^{+}}}\right\} 
\left(
\|u_0\|_{L^{\overline{q}} (\Rt)}
    +
\| \vf\|_{L^1_t(L^{\overline{q}}_x)} 
\right)
\\ &
=   CR^b  \| u \|_{E_T} +
B. 
    \end{split}
\end{equation}
At this point, we must stress that provided with $2CB \leq \frac{1}{2}$, we get $CR^b \leq  \frac{1}{2} $,  then considering \eqref{26fevr1} and \eqref{26fevr2}, we obtain
\begin{equation*}
        \|\Phi (u) - \Phi ( v) \|_{E_T} 
\leq
\frac{1}{2}
 \|
    u
    -
     v \|_{
 E_T
    }
\quad \text{and} \quad
        \|\Phi (u)
    \|_{
E_T} 
\leq
R. 
\end{equation*}
Thus, considering the contraction mapping principle (see Proposition \ref{cmp}), it follows that there exists a unique local-in-time mild solution $u \in  E_T$ of the equation \eqref{NS_Intro}. 
The proof of Theorem \ref{Theoreme_1} is  concluded.


\paragraph{\bf Acknowledgements.} 
The author warmly thanks Pierre-Gilles Lemarié-Rieusset and Diego Chamorro for their helpful comments and advises. The author is supported by the ANID postdoctoral program BCH 2022 grant No. 74220003.

	
	\bibliographystyle{siam}
	\bibliography{pdBIBLIO}

\end{document}